\documentclass[11pt]{article}
\usepackage{amsmath, amssymb, epsfig,amsthm}
\usepackage[all]{xy}
\numberwithin{equation}{section}

\newtheorem{theorem}{Theorem}[section]
\newtheorem{proposition}[theorem]{Proposition}

\newtheorem{conjecture}[theorem]{Conjecture}
\newtheorem{lemma}[theorem]{Lemma}

\newtheorem{problem}[theorem]{Problem}

\theoremstyle{definition}
\newtheorem{definition}[theorem]{Definition}

\newtheorem{remark}[theorem]{Remark}

\newcommand{\skel}{\mathrm{Skel}}
\newcommand\lk{\mathrm{lk}}
\newcommand\st{\mathrm{st}}
\newcommand{\sd}{\mathrm{sd}}
\newcommand{\field}{{\bf k}}

\title{Face enumeration on flag complexes and flag spheres}
\author{Hailun Zheng\\
	\small Department of Mathematics \\[-0.8ex]
	\small University of Michigan\\[-0.8ex]
	\small Ann Arbor, MI 48109-1043, USA\\[-0.8ex]
	\small \texttt{hailunz@umich.edu}
}
\begin{document}
\maketitle
\begin{abstract}
	We give a survey on the recent results and problems on the face enumeration of flag complexes and flag simplicial spheres, with an emphasis on the characterization of face vectors of flag complexes, several lower-bound type of conjectures including the Charney-Davis conjecture and Gal's conjecture, and the upper bound conjecture for flag spheres and pseudomanifolds.
\end{abstract}
\section{Introduction}
In this paper we will survey some fascinating properties of flag complexes and results on the face enumeration of flag simplicial spheres. A simplicial complex is called \emph{flag} if all of the minimal non-faces have cardinality two, or equivalently, it is the clique complex of its graph.

We begin with a summary of major face enumeration results for general simplicial complexes and simplicial spheres. The Kruskal-Katona theorem \cite{Katona,Kruskal} fully characterizes the $f$-vector of simplicial complexes. In 1971, McMullen \cite{McMullen-UBT} conjectured a list of conditions to characterize the $f$-vector of simplicial polytopes. Around 1980, the work of Billera and Lee \cite{Billera-Lee-80,Billera-Lee-81} showed by construction the sufficiency of McMullen's conditions, and Stanley \cite{Stanley-g-theorem} proved the necessity, thus establishing the classical $g$-theorem. It is conjectured that the same characterization even holds for the $f$-vector of simplicial spheres. Other major results on face enumeration include the Upper Bound Theorem \cite{McMullen-UBT} (UBT, for short) and the Lower Bound Theorem \cite{Barnette-LBT} (LBT, for short), which state that among all simplicial $d$-polytopes with $n$ vertices, cyclic polytopes simultaneously maximize all the face numbers while stacked polytopes simultaneously minimize all the face numbers. Both of the theorems extend to the class of simplicial spheres \cite{Barnette-LBT,Stanley-UBT}, and have various generalizations in even larger classes of simplicial complexes, see, for example, \cite{Fogelsanger, Kalai-rigidity, Novik-UBT}.

However, none of the above results give tight bounds on the face numbers of flag complexes or spheres. For example, the clique complex of the graph of a stacked $d$-polytope and that of a cyclic $d$-polytope with $n$ vertices are $d$-dimensional and $(n-1)$-dimensional respectively, and hence they cannot be flag $(d-1)$-spheres. It is natural to ask if there are flag analogs of the Kruskal-Katona theorem, the $g$-theorem, the UBT and the LBT, etc. 

We know very little about the answers to these problems in general. In fact, we don't even have a plausible conjecture for the flag Kruskal-Katona theorem. The difficulty comes from the fact that the $f$-vector of a flag complex cannnot be changed ``continuously": adding one edge to the graph of a given flag complex may result in a huge change in the $f$-vector of the corresponding clique complex. 
As to the flag lower bounds on the face numbers, various conjectures have been proposed in the last decade. Here we just mention one remarkable contribution of Gal: in \cite{Gal-real root conjecture} he defined the $\gamma$-vector which can be considered as the flag analog of the $g$-vector for general simplicial spheres, and he further conjectured that the $\gamma$-numbers are nonnegative for flag simplicial spheres. This conjecture has been verified in several important classes of flag simplicial spheres, including all 3-dimensional flag spheres \cite{DavisOkun-dimension 3 CD conjecture} and the barycentric subdivision of all simplicial spheres (even of regular CW spheres) \cite{Karu-cd-index}.

A upper bound conjecture (UBC, for short) for flag spheres was first proposed for odd-dimensional ones \cite{Nevo-Lutz, Nevo-Petersen}. It states that among all flag $(2k-1)$-spheres with $n$ vertices (or more generally flag manifolds), the join of cycles of length as equal as possible is the unique maximizer of all the face numbers. This conjecture has been verified asymptotically \cite{Adamaszek-Hladky} and for the edge number \cite{Zheng-flag edge number}. The even-dimensional flag UBC \cite{Zheng-flag 3-dimensional} turns out to be more complicated. As we can see from the dimension 4 case (where Gal's result gives the upper bounds of the face numbers), the maxmizer is at least not unique. So far this conjecture is wide open.

The paper is structured as follows. In Section 2 we give basic definitions, list several important properties of flag complexes, and provide the preliminary results on the lower bounds on $f$-, $h$- and $g$-vectors. in Section 3 we discuss how the $f$-vectors of flag complexes behave. Section 4 and Section 5 are devoted to the flag LBT and UBT respectively: in Section 4 we will see how the geometric meaning of flagness and introducing the $cd$-index help proving several important cases of the flag lower bound conjecture. We close with surveying current results on the flag UBC in Section 5.
 
\section{Preliminaries}
\subsection{Definitions and Properties of flag complexes}
A \emph{simplicial complex} $\Delta$ on a vertex set $V=V(\Delta)$ is a collection of subsets
$\sigma\subseteq V$, called faces, that is closed under inclusion. For $\sigma\in \Delta$, let $\dim\sigma:=|\sigma|-1$ and define the \emph{dimension} of $\Delta$, $\dim \Delta$, as the maximal dimension of its faces. A \emph{facet} is a maximal face under inclusion. We say $\Delta$ is \emph{pure} if all of its facets have the same dimension. If $\Delta$ is a simplicial complex and $\sigma$ is a face of $\Delta$, the \emph{link} of $\sigma$ in $\Delta$ is $\lk(\sigma,\Delta):=\{\tau-\sigma\in \Delta: \sigma\subseteq \tau\in \Delta\}$. When the context is clear, we will abbreviate the notation and write it as $\lk(\sigma)$. If $W$ is a subset of $V(\Delta)$, the \emph{induced} subcomplex of $\Delta$ on $W$ is the subcomplex $\Delta[W]=\{\tau\in\Delta: \tau\subseteq W\}$.

A simplicial complex $\Delta$ is a \emph{simplicial sphere} (\emph{simplicial manifold}, resp.) if the geometric realization of $\Delta$, denoted as $||\Delta||$, is homeomorphic to a sphere (manifold, resp). A simplicial sphere is \emph{polytopal} if it can be realized as the boundary complex of a polytope. Let $\field$ be a field and let $\tilde{H}_∗ (\Gamma, \field)$ denote the reduced singular homology of $||\Gamma||$ with coefficients in $\field$. For a pure $(d-1)$-dimensional simplicial complex $\Delta$ and a field $\field$, we say that $\Delta$ is a $\field$-\emph{homology manifold} if $\tilde{H}_∗ (\lk(\sigma), \field)\cong\tilde{H}_∗ (\mathbb{S}^{d-1-|\sigma|}, \field)$ for every nonempty face $\sigma\in\Delta$. A $\field$-\emph{homology sphere} is a $\field$-homology manifold that has the $\field$-homology of a sphere. Every simplicial manifold (resp. simplicial sphere) is a homology manifold (resp. homology sphere). A $(d-1)$-dimensional simplicial complex $\Delta$ is called a $(d-1)$-\emph{pseudomanifold} if it is pure and every $(d-2)$-face (called ridge) of $\Delta$ is contained in exactly two facets. A $(d-1)$- pseudomanifold $\Delta$ is called a \emph{normal $(d-1)$-pseudomanifold} if it is connected, and the link of each face of dimension $\leq d- 3$ is also connected. For a fixed $d$, we have the following hierarchy:

simplicial $(d-1)$-spheres $\subseteq$ homology $(d-1)$-spheres $\subseteq$ homology $(d- 1)$-manifolds

{\centering $\subseteq$ normal $(d- 1)$-pseudomanifolds $\subseteq$ $(d-1)$-pseudomanifolds.}

When $d = 3$, the first two classes complexes above coincide; so do the third and fourth classes. However, starting from $d = 4$, all the inclusions above are strict.

A simplicial complex is called \emph{flag} if all of its minimal non-faces have cardinality two, or equivalently, it is the clique complex of its graph. Flag complexes have many properties which general simplicial complexes do not usually have, as we shall see in the following lemma \cite[Lemma 5.2]{Nevo-Petersen}.
\begin{lemma}\label{lm: face link prop}
		Let $\Delta$ be a flag complex on vertex set $V$.
		\begin{itemize}
			\item If $W\subseteq V$, then $\Delta[W]$ is also flag. 
			\item If $\sigma$ is a face in $\Delta$, then $\lk(\sigma)=\Delta[V(\lk(\sigma))]$. In particular, all face links in a flag complex are also flag.
			\item Any edge $\{v,v'\}$ in $\Delta$ satisfies the link condition $\lk(v)\cap \lk(v')=\lk(\{v,v'\})$. More generally, any face $\sigma=\sigma_1\cup \sigma_2$ in $\Delta$ satisfies $\lk(\sigma)=\lk(\sigma_1)\cap \lk(\sigma_2)$.
		\end{itemize}
\end{lemma}
The next property is about the connectivity of the graph of a complex. The graph of a simplicial complex $\Delta$ is said to be
\emph{$n$-connected} if $\Delta$ has at least $n+1$ vertices and the complex obtained from $\Delta$ by deleting
any $n-1$ or fewer vertices and their incident faces is connected. It is known that every polytopal $(d-1)$-sphere, or more generally, every $(d-1)$-pseudomanifold is $d$-connected, see \cite[theoerem 3.14]{Ziegler-book} and \cite[Corollary 5]{Barnette-pseudomanifold connectedness}. For flag complexes, a stronger statement holds \cite[Theorem 1.2]{Athanasiadis-flag prop}.
\begin{lemma}\label{lm: graph conn prop}
	For every flag simplicial pseudomanifold $\Delta$ of dimension $d-1$, the graph
	$G(\Delta)$ is $(2d-2)$-connected.
\end{lemma}

We introduce several ways to construct new flag complexes from given ones. If $\Delta$ and $\Gamma$ are two simplicial complexes on disjoint vertex sets, then the \emph{join} of $\Delta$ and $\Gamma$, denoted as $\Delta *\Gamma$, is the simplicial complex on vertex set $V(\Delta)\cup V(\Gamma)$ whose faces are $\{\sigma\cup\tau:\sigma\in\Delta, \tau\in\Gamma\}$. Hence by the definition, if $\Delta$ and $\Gamma$ are flag complexes of dimension $d_1-1$ and $d_2-1$ respectively, then $\Delta*\Gamma$ is a flag complex of dimension $d_1+d_2-1$. 

To simplify the notation, in the following we write the one-vertex set $\{v\}$ as $v$. For a simplicial complex $\Delta$ and a face $\sigma$ in it, let the \emph{stellar subdivision} of $\Delta$ at $\sigma$ be 
\[\sd(\sigma,\Delta)=\{\tau\in\Delta: \tau\cap \sigma=\emptyset\}\cup (v*\partial \sigma *\lk(\sigma,\Delta)),\]
where $v$ is a new vertex. The stellar subdivision of $\Delta$ is always PL homeomorphic to $\Delta$. However, since the stellar subdivision at face $\sigma$ creates a missing face $\sigma$, $\sd(\sigma,\Delta)$ is flag only when $\sigma$ is an edge. The inverse of edge subdivision is called the edge contraction:
\[\sd^{-1}(\{u,v\}, \Delta)=\{\tau\in\Delta:u\in\tau\}\cup\{(\tau\cup v)\backslash u: u\in F\in\Delta\}.\]
It is well-known that the edge contraction of a flag complex $\Delta$ is flag if and only if the contracted edge does not belong to any induced 4-cycles in $\Delta$. The following lemma \cite[Theorem 1.2]{Nevo-Lutz} is the flag analog of the classical Alexander theorem in PL topology.
\begin{lemma}
	Two flag simplicial complexes are PL homeomorphic if and only if they can be connected by a sequence of edge subdivisions and their inverses
	such that all the complexes in the sequence are flag.
\end{lemma} 
\subsection{The lower bounds on $f$-, $h$- and $g$-vectors}
For a $(d-1)$-dimensional complex $\Delta$, we let $f_i = f_i(\Delta)$ be the number of $i$-dimensional faces of $\Delta$ for $-1\leq i\leq d-1$. The vector $(f_{-1}, f_0, \ldots, f_{d-1})$ is called the $f$\emph{-vector} of $\Delta$. The $h$-vector of a $(d-1)$-dimensional simplicial complex $\Delta$, $h(\Delta)=(h_0(\Delta), h_1(\Delta),\dots, h_{d-1}(\Delta))$, is defined by the equality
\[\sum_{i=0}^d h_i(\Delta)t^i=\sum_{i=0}^d f_{i-1}(\Delta)t^i(1-t)^{d-i}.\] 
The polynomial which appears in the left-hand side of the above identity is the $h$-polynomial of $\Delta$ and is denoted by $h_\Delta(x)$. Define the $g$-vector $g(\Delta)=(g_0(\Delta), g_1(\Delta), \dots, g_{\left\lfloor d/2\right\rfloor}(\Delta))$, where $g_0(\Delta)=1$ and $g_i=h_i(\Delta)-h_{i-1}(\Delta)$ for $i\neq 0$. It is not hard to see that among all flag simplicial $(d-1)$-spheres, the octahedral $(d-1)$-sphere simultaneously minimizes all $f$-numbers. This even holds in the class of flag pseudomanifolds \cite[Proposition 2.2]{Athanasiadis-flag prop}.
\begin{proposition}\label{lm: lower bound f-numbers}
	Let $\Delta$ be a $(d-1)$-dimensional flag pseudomanifold. Then $f_{i-1}(\Delta)\geq 2^i\binom{d}{i}$ for all $0\leq i\leq d$. Equality holds if and only if $\Delta$ is the octahedral $(d-1)$-sphere.
\end{proposition}

For any flag complex $\Delta$ and vertex $v\in\Delta$, we have $h_i(\Delta)=h_i(\Delta\backslash v)+h_{i-1}(\lk(v))$ whenever $\dim(\Delta)=\dim(\Delta\backslash v)$. this identity together with the monotonicity of the $h$-vectors of simplicial spheres \cite{Stanley-monotonicity} yields the lower bounds on the $h$-numbers \cite[Theorem 1.3]{Athanasiadis-flag prop}.
\begin{proposition}\label{lm: lower bound h-numbers}
	Let $\Delta$ be a flag homology $(d-1)$-sphere. Then $h_i(\Delta)\geq \binom{d}{i}$.
\end{proposition}
\begin{remark}
	There is a generalization of the above lemma for flag Buchsbaum$^*$ complexes, see Proposition 5.5 and Section 6 in \cite{Athanasiadis-Buchsbaum}.
\end{remark}
Our next goal is to obtain lower bounds on the $g$-numbers. For $g_2$, this result is likely known to the experts. However, it seems to be missing from the literature. In the following we prove the lower bound on $g_2$ using the rigidity theory. Let $\field$ be an infinite field and assume that $V(\Delta)=[n]$. The \emph{Stanley-Reisner ring} of $\Delta$ is $\field[\Delta]=\field[x_1,x_2,\dots, x_n]/I_\Delta$, where the \emph{Stanley-Reisner ideal} is \[I_\Delta=(x_{i_1}\dots x_{i_k}:\{i_1,\dots, i_k\}\notin \Delta).\]
A $(d-1)$-dimensional complex is \emph{$\field$-rigid} if for generic linear forms $\theta_1, \dots, \theta_{d+1}$ and $1\leq i\leq d+1$, the multiplication map $\cdot \theta_i: \field[\Delta]/(\theta_1, \dots, \theta_{i-1})_1 \to \field[\Delta]/(\theta_1, \dots, \theta_{i-1})_2$ is injective. For this interpretation of rigidity, see \cite{Lee-stress space}. We also refer to \cite{Kalai-rigidity, Stanley-green-book} for basics in the rigidity theory of frameworks and Stanley-Reisner ring respectively.
\begin{lemma}\label{lm: k-rigidity of complex minus a pt}
	Let $d\geq 4$ and $\Delta$ be a flag normal $(d-1)$-pseudomanifold. Then for any vertex $v$, $\Delta \backslash v$ is $\field$-rigid.
\end{lemma}
\begin{proof}
	First we claim that $\Delta \backslash v=\cup_{u\notin\st(v)} \st(u)$. Indeed, if there is a facet $F$ of $\Delta\backslash v$ with $V(F)\subseteq V(\lk(v))$, then by the flagness of $\Delta$, $F\cup v\in\Delta$, a contradiction. Hence every facet of $\Delta \backslash v$ must contain at least one vertex from $\Delta\backslash v-\lk(v)$, i.e., $\Delta \backslash v\subseteq\cup_{u\notin \st(v)} \st(u)$. The other inclusion is obvious.
	
	Since $\lk(v)$ is a normal $(d-2)$-pseudomanifold, it is $\field$-rigid. By the cone lemma \cite[Lemma 5.3]{Novik-Swartz}, $\st(v)$ is also $\field$-rigid. Furthermore, the vertices of $\Delta\backslash v-\lk(v)$ can be ordered as $(u_1,u_2,\dots)$ so that the induced graph on any initial segment is connected, and hence $\st(u_i) \cap (\cup_{j<i} \st(u_j))$ contains a facet of $\Delta\backslash v$. So by the gluing lemma \cite[Lemma 5.4]{Novik-Swartz}, $\Delta \backslash v$ is $\field$-rigid.
\end{proof}
\begin{proposition}\label{lm: lower bound g2}
	Let $d\geq 3$ and $\Delta$ be a flag normal $(d-1)$-pseudomanifold. Then $g_2(\Delta)\geq \binom{d}{2}-d$. Furthermore when $d\geq 4$, $g_2(\Delta)=\binom{d}{2}-d$ if and only if $\Delta$ is the octahedral $(d-1)$-sphere.
\end{proposition}
\begin{proof}
	We prove by induction on the dimension. Let $\field$ be an infinite field of characteristic zero. For $d=3$, $g_2(\Delta)\geq 0$ follows from the fact that all triangulated 2-manifolds are $\field$-rigid. Let $d\geq 4$. For any vertex $v$ of $\Delta$, $\dim \Delta \backslash v=\dim \Delta$. By Lemma 4.1 in \cite{Athanasiadis-flag prop}, $h_i(\Delta)=h_i(\Delta \backslash v)+h_{i-1}(\lk_\Delta v)$ for $i=1,2$, and hence $g_2(\Delta)=g_2(\Delta\backslash v)+g_1(\lk(v))$. Since $\Delta$ is a flag normal pseudomanifold, there is a vertex $u$ such that $\{u,v\}\notin \Delta$. Let $\omega\in \field[x_1,\dots, x_{f_0}]_1$ and $\Theta$ a generic linear system of parameters for $\Delta$. Note that both $\Delta\backslash v$ and $\st(u)$ are induced subcomplexes of $\Delta$ and furthermore $\Delta\backslash v\supseteq \st(u)$. So there is a surjection $p: \field[\Delta \backslash v] \to \field[\st(u)]\cong \field[\lk(u)]$. In the following let $\field(\Gamma)=\field[\Gamma]/\Theta$ for $\Gamma=\Delta \backslash v$ or $\lk(u)$. The rigidity theory gives us the following commutative diagram.
	\[
	\xymatrix{
		\field(\Delta \backslash v)_1 \ar@{^{(}->}[r]^{\phi_1=\cdot \omega} \ar@{->>}[d]^{\bar{p}_1} & \field(\Delta \backslash v)_2 \ar[r] \ar@{->>}[d]^{\bar{p}_2} & {\rm{coker}}\phi_1 \ar[r] \ar[d]^{q} & 0 \\
		\field(\lk(u))_1 \ar@{^{(}->}[r]^{\phi_2=\cdot \omega}  & \field(\lk(u))_2 \ar[r] & {\rm{coker}}\phi_2 \ar[r]  & 0.
	}
	\]
	Furthemore, since $\Delta\backslash v$ is $\field$-rigid, \[\dim_\field {\rm{coker}}\phi_1=\dim_\field \field(\Delta \backslash v)_2-\dim_\field \field(\Delta\backslash v)_{1}=h_2(\Delta\backslash v)- h_1(\Delta\backslash v)=g_2(\Delta \backslash v).\] Similarly, $\dim_\field {\rm{coker}}\phi_2=g_2(\lk(u))$. The fact that the above diagram commutes implies that $q$ is also surjective. Hence by the inductive hypothesis, $g_2(\Delta\backslash v)\geq g_2(\lk(u))\geq \binom{d-1}{2}-(d-1)$. Notice that $g_1(\lk(v))\geq d-2$ for every vertex $v$, so $g_2(\Delta)\geq \binom{d-1}{2}-(d-1)+(d-2)=\binom{d}{2}-d$. When $d\geq 4$, the equality holds if and only if the $g_2$ of every vertex link is exactly $d-2$, i.e., it is octahedral. Hence $\Delta$ is the octahedral sphere.
\end{proof}

What about the lower bounds on the other $g$-numbers? As we will see in Section 4, the following conjecture would follow if the nonnegative of the $\gamma$-vectors holds. However, even this conjecture is open.
\begin{conjecture}
	Let $\Delta$ be a flag homology $(d-1)$-sphere. Then $g_i(\Delta)\geq \binom{d}{i}-\binom{d}{i-1}$. Equality holds if and only if $\Delta$ is the boundary complex of the $d$-cross-polytope.
\end{conjecture}
\section{Face vectors of flag complexes}
Before we dive into the face vectors of flag spheres or manifolds, let's warm up by discussing the face vectors of flag complexes without placing any conditions on the geometric realization. For positive integers $n$ and $i$, $n$ can be expressed uniquely in the form
\[n=\binom{n_i}{i}+\binom{n_{i-1}}{i-1}+\cdots+\binom{n_j}{j},\] where $n_i>n_{i-1}>\dots >n_j\geq j\geq 1$. Define 
\[n^{(i)}=\binom{n_i}{i+1}+\binom{n_{i-1}}{i}+\cdots+\binom{n_j}{j+1},\]
where $\binom{k}{l}=0$ if $l>k$. Also we say a simplicial complex $\Delta\subseteq 2^{[n]}$ is \emph{compressed} if its set of $k$-faces forms an initial segment with respect to the reverse lexicographic order on the $(k + 1)$-subsets of $[n]$ for each $k$.  Now we are ready to state the Kruskal-Katona theorem \cite{Katona,Kruskal} which characterizes the $f$-vector of simplicial complexes. 
\begin{theorem}
	A vector $f=(f_{-1},f_0,f_1,\dots, f_{d-1})$ is the $f$-vector of a $(d-1)$-dimensional simplicial complex if and only if $f_{-1}=1$ and
	\[f_{i+1}\leq f_i^{(i+1)}, \quad 0\leq i\leq d-2.\]
	Moreover $f$ can be realized as the $f$-vector of a compressed simplicial complex.
\end{theorem}

Is there a Krukal-Katona type of theorem for flag complexes? Alternatively, given a graph with a fixed number of $k$-cliques, how many
$(k + 1)$-cliques can it have? Unlike the simplicial complexes, where we can add one face at a time and change the $f$-vector gradually, adding a single edge to a flag complex may cause the face vector of the resulting flag complex to vary dramatically. So it is natural to ask how many $i$-faces a flag $(d-1)$-dimensional complex with $n$ vertices can have. The celebrated T\'uran's theorem states that given any graph $G$ with $n$ vertices and no cliques of size $r+1$, the T\'uran graph has the maximum number of edges. There are many generalization of this theorem regarding the number of cliques in graphs, see, for example, \cite{Eckhoff}. Another interesting conjecture which gives restriction on the range of $f$-vectors was proposed by Kalai (unpublished, see \cite{Stanley-green-book}) and Eckhoff \cite{Eckhoff-conj} independently. Frohmader established it as a theorem \cite{Frohmader-Kalai's conjecture}. Recall that a $(d-1)$-dimensional simplicial complex is called balanced if its graph is $d$-colorable. 
\begin{theorem}
	The $f$-vector of a flag complex is the $f$-vector of a certain balanced complex. 
\end{theorem}
As a corollary, the $f$-vector of any flag complex must satisfy the Frankl-F\"uredi-Kalai inequalities \cite{Frankl-Furedi-kalai}. However, the above theorem does not give a full characterization of the face vectors of flag complex. In fact, there exist many vectors that are $f$-vectors of balanced complexes but not of flag complexes. How to describe those non-admissible $f$-vectors remains unknown. We also remark that the constructive proof given by Frohmader doesn't solve another conjecture of Kalai, which is a stronger version of the above statement.
\begin{conjecture}\label{conj: flag CM}
	The $f$-vector of any flag Cohen-Macaulay complex is the $f$-vector of a balanced Cohen-Macaulay complex.
\end{conjecture}
We end this section with a conjecture \cite{Constantinescu-Varbaro} that implies Conjecture \ref{conj: flag CM}. Some partial results were obtained in \cite{CCV, Cook-Nagal-flag CM}.
\begin{conjecture}
	The set of $h$-vectors of flag Cohen-Macaulay simplicial complexes equals the set of $f$-vectors of flag simplicial complexes.
\end{conjecture}

\section{The flag lower bound conjecture}
In this section we will discuss various lower-bound type conjectures. Among them, the Charney-Davis conjecture comes first in the history. The geometric intuition behind this conjecture is from a long-standing conjecture in differential geometry.
\begin{conjecture}\label{conj: Hopf conjecture}
	If $M^{2n}$ is a closed Riemannian manifold of nonpositive sectional curvature, then $(-1)^{2n} \chi(M^{2n})\geq 0$.
\end{conjecture}
To find a discrete version of the above conjecture, note that a simplicial complex $\Delta$ can be turned into a geodesic space by letting all the edges have length $\pi/2$. Gromov \cite{Gromov} showed that in this setting a cubical complex is non-positively curved if and only if the links are flag. (Gromov's interpretation of flagness is a very useful tool in studying flag simplicial complexes. For example, it is shown in \cite{Adiprasito-Benedetti} that the Hirsch conjecture holds for all flag homology manifolds by using the fact that every vertex star in the flag complex is geodesically convex.) Together with the fact that the Euler characteristic can be interpretted as $\chi(\Delta)= 1+ \sum_i (-1/2)^{i+1} f_i(\Delta)$, Conjecture \ref{conj: Hopf conjecture} is rephrased as the Charney-Davis conjecture in \cite{CharneyDavis-CD conjecture}:
\begin{conjecture}\label{conj: Charney-Davis conjecture}
	Let $\Delta$ be a flag simplicial $(2m-1)$-sphere (or more generally, a flag homology $(2m-1)$-sphere). Then $(-1)^m h_\Delta(-1)\geq 0$.
\end{conjecture}
In the case of $m=2$, the conjecture says that $f_1(\Delta)-5f_0(\Delta)+16\geq 0$, which is proved in \cite{DavisOkun-dimension 3 CD conjecture} by heavy topological machinery. In higher dimensions the conjecture remains open in general. What about even-dimensional flag spheres? Define a polynomial $\tilde{h}_\Delta(t):=\frac{1}{1+t}h_\Delta(t)$. Gal and Januszkiewicz \cite{Gal-J-CD conjecture} noticed that the Charney-Davis conjecture is equivalent to the following conjecture.
\begin{conjecture}
Let $\Delta$ be a flag simplicial $2m$-sphere (or more generally, a flag homology $2m$-sphere). Then $(-1)^m \tilde{h}_\Delta(-1)\geq 0$.
\end{conjecture}
In the next two subsections we will discuss the proof of the Charney-Davis conjecture in some special cases by introducing the $\gamma$-vector and the $cd$-index.
\subsection{The lower bound conjectures around the $\gamma$-vector}
We start by discussing the classic Generalized Lower Bound Theorem (GLBT, for short) without a requirement that a polytope or sphere is flag. It turns out that instead of the $f$-vector, GLBT is better stated in terms of the $g$-vector.
\begin{theorem}\label{thm: GLBT}
Let $P$ be a simplicial $d$-polytope. Then 
\begin{enumerate}
    \item $g_i(P)\geq 0$ for $1\leq i\leq d/2$.
    \item If $g_r(P)=0$ for some $1\leq r \leq d/2$, then $P$ is $(r-1)$-stacked, i.e., there is a triangulation $K$ of $P$ all of whose faces of dimension at most $d-r$ are faces of $P$.
\end{enumerate}
\end{theorem}
In fact we even know what the $g$-numbers count: as discovered by Lee \cite{Lee-stress space}, each $g_r$ is the dimension of the affine stress spaces associated with the polytope.

Since $(i-1)$-stacked $d$-polytopes are not flag, the lower bounds given by Theorem \ref{thm: GLBT} are not tight for flag polytopes or spheres. One would expect that there is an analogous vector for flag spheres that plays the same role as the $g$-vector for general spheres. By the Dehn-Sommerville relations, the $h$-vector of any flag homology $(d-1)$-sphere $\Delta$ is symmetric. Hence the $h$-polynomial of $\Delta$ has an integer expansion in the basis $\{t^i(1+t)^{d-2i}: 0\leq i\leq d/2\}$, which shows that the following vector is well-defined.
\begin{definition}\label{df: gamme vector}
The \emph{$\gamma$-vector} of a flag homology $(d-1)$-sphere $\Delta$ is the vector $\gamma(\Delta)=(\gamma_0,\dots, \gamma_{\left\lfloor\frac{d}{2}\right\rfloor})$, whose entries are given by
\[h_\Delta(t)=\sum_{0\leq i\leq d/2}\gamma_i(\Delta)t^i(1+t)^{d-2i}.\]
The polynomial $\gamma_\Delta (t)=\sum_{0\leq i\leq d/2} \gamma_i(\Delta) t^i$ is called the \emph{$\gamma$-polynomial} of $\Delta$. Alternatively, the $\gamma$-polynomial can be defined as a polynomial of degree at most $\left\lfloor\frac{d}{2}\right\rfloor$ such that
\[h_\Delta(t)=(1+t)^d\gamma_\Delta\left(\frac{t}{(1+t)^2}\right).\]
\end{definition}
By Lemma \ref{lm: lower bound f-numbers}, $\gamma_1=f_0-2d\geq 0$. Also $\gamma_2=f_1-(2d-3)f_0+2d(d-2)$ is known to be nonnegative for $d=3$ by the Davis-Okun theorem \cite{DavisOkun-dimension 3 CD conjecture}. Gal proposed the following conjecture \cite[Conjecture 2.1.7]{Gal-real root conjecture}. 
\begin{conjecture}\label{conj: Gal's conjecture}
If $\Delta$ is a flag homology sphere then $\gamma(\Delta)$ is nonnegative.
\end{conjecture}
In particular, Charney-Davis conjecture states that $\gamma_{\left\lfloor d/2\right\rfloor}\geq 0$ and is a special case of Gal's conjecture. Conjecture \ref{conj: Gal's conjecture} is known to hold for several important classes of flag spheres, including the order complexes of Gorenstein$^*$ complexes and Coxeter complexes. In an attempt of finding a flag analog of stackness, Nevo and Lutz proposed the following conjecture \cite[Conjecture 6.1]{Nevo-Lutz}, which is still open.
\begin{conjecture}\label{conj: gamma_2=0}
	Let $d\geq 4$ be an integer and $\Delta$ be a flag simplicial $(d-1)$-sphere.
	Then the following are equivalent:
	\begin{enumerate}
		\item $\gamma_2(\Delta)=0$.
		\item There is a sequence of edge contractions from $\Delta$ to the boundary of the $d$-cross-polytope such that all complexes in the sequence are flag spheres, and the link of each edge
		contracted is the octahedral $(d-3)$-sphere.
	\end{enumerate}
\end{conjecture}
However, we don't have a conjecture of which flag spheres can attain $\gamma_i=0$ for $i\geq 3$. On the other hand, the $\gamma$-vector has an explicit combinatorial description in some cases (e.g., the $\gamma_i$ of type A, B and D Coxeter complexeres are interpretted as the number of permutations in $S_n$ that satisfies certain conditions, see \cite{Nevo-Petersen}). We propose the following problems:
\begin{problem}\label{prob: meaning of gamma numbers}
	
	\begin{enumerate}
		\item Find a characterization of flag homology $(d-1)$-spheres with $\gamma_i=0$ for $3\leq i\leq d/2$, 
		\item Find a combinatorial interpretation of the $\gamma$-numbers.
	\end{enumerate}

\end{problem} 
We end this section with a strenghening of Conjecture \ref{conj: Gal's conjecture} suggested by Nevo and Petersen \cite{Nevo-Petersen}. 
\begin{conjecture} Let $\Delta$ be a flag homology sphere.
\begin{enumerate}
	\item $\gamma(\Delta)$ satisfies the
	Kruskal-Katona inequalities.

	\item $\gamma(\Delta)$ satisfies the
	Frankl-F\"uredi-Kalai inequalities.

\end{enumerate}
\end{conjecture}
The first part of the conjecture is true but not sharp for flag 3- (or 4)-spheres. The second conjecture is more interesting; it is confirmed for barycentric subdivisions of simplicial homology spheres in \cite{Nevo-Petersen-Tancer}, using enumerative properties of Eulerian polynomials. Later it is also proved for barycentric subdivisions of the boundary complexes of polytopes in \cite{Murai-Nevo} via $cd$-index.

\subsection{Proof idea: the $cd$-index}
In this section, we will discuss Karu's break-through result on the nonnegativity of the $\gamma$-vector. 
Let $P$ be a finite graded poset of rank $n+1$ with $\hat{0}$ and $\hat{1}$, and the rank function is $\rho: P\to \{0,1,\dots, n+1\}$. For any $S\subseteq [n]$, the \emph{$S$-rank selected subposet} of $P$ is
\[P_S=\{x\in P: \rho(x)\in S\}\cup\{\hat{0}, \hat{1}\}.\] The function $\alpha: 2^{[n]}\to \mathbb{Z}$, where each $\alpha(S)$ counts the number of maximal chains in $P_{S}$, is called the \emph{flag $f$-vector} of $P$. Also define the \emph{flag $h$-vector} of $P$ as the function $\beta: 2^{[n]}\to \mathbb{Z}$,
\[\beta(S)=\sum_{T\subset S}(-1)^{|S|-|T|}\alpha(T).\]

The $cd$-index is a polynomial in variables $c,d$ which encodes the flag $h$-vector of $P$. We assign each $S\subseteq [n]$ a noncommutative polynomial $u_S=u_1u_2\dots u_n$ in the variables $a,b$ by
\[u_i=
\begin{cases}
a & i\notin S\\
b & i\in S
\end{cases}
.\]
Let $\psi_P(a,b)=\sum_{S\subseteq [n]} \beta(S)u_S$. In particular, when $P$ is an Eulerian poset, $\psi_P(a,b)$ can further be written as a polynomial $\phi_P(c,d)$ in $c=a+b$ and $d=ab+ba$. This is the \emph{$cd$-index} of $P$.

It is conjectured that the $cd$-index of a regular CW sphere, or more generally, a Gorenstein$^*$ poset (see \cite{Stanley-cd-index} for the definition), is nonnegative. Many earlier works around this conjecture were based on recursive formulas of the $cd$-index, see, for example, \cite{Stanley-cd-index}. We remark that it was also observed in \cite{Ehrenborg-Readdy} that the $cd$-index is a coalgebra map from the vector space spanned by isomorphism
classes of Eulerian posets to the algebra of polynomials in the non-commutative variables $c,d$, which provides another tool to derive formulas for the $cd$-index. Karu gave a full proof of the conjecture in \cite{Karu-cd-index} by algebraic geometry tools. 
\begin{theorem}\label{thm: Karu's result} 
	If $\Delta$ is a Gorenstein$^*$ poset, then $\phi_\Delta(c,d)\geq 0$, i.e., the coefficients of $\phi_\Delta(c,d)$ are nonnegative.
\end{theorem}
We sketch Karu's proof in the special case when $\Delta$ is a complete $n$-dimensional shellable fan. To each $cd$-monomial $M$ we associate a polynomial in $t_1,\dots, t_n$ by replacing $c$ with $1+t_i$ and $d$ with $t_i+t_{i+1}$; for example, $cd=(t_1+1)(t_2+t_3)$. Let $P_n(\Delta)=\sum_S h_S(\Delta) t^S$ be the Poincar\'e polynomial of $\Delta$. Let $\sigma_1,\dots, \sigma_n$ be a shelling of $\Delta$ and define $\sigma_i^-=\sigma_i\cap (\cup_{j=1}^{i-1}\sigma_j)$. The $(n-1)$-skeleton of $\Delta$ can be constructed from $\partial \sigma_n$ by attaching $\sigma_{n-1}^-, \sigma_{n-2}^-, \dots, \sigma_2^-$ in order. By Poincar\'e duality, if $P_{n-2}(\partial \sigma_i^-)=g_i(c,d)$ for some $cd$-polynomial of degree $n-2$, then $P_{n-2}(\sigma_i^-, \partial \sigma_i^-)=f_i(c,d)+g_i(c,d)t_{n-1}$ for some $cd$-polynomial $f_i$ of degree $n-1$. So we have that
\[P_n(\skel_{n-1}(\Delta))=\sum_{i=2}^n P_{n-2}(\sigma_i^-, \partial \sigma_i^-)=\sum_{i=2}^n f_i+g_it_{n-1}.\] 
From this one can show that $P_n(\Delta)=\sum_{i=2}^n f_ic+g_id$, and then the proof is done by induction on the dimension. The proof of the general case is more subtle and requires the sheaf theory, see \cite{Karu-cd-index} for the details.

The $cd$-index is closely related to the $\gamma$-vector. Let $\sd(\Delta)$ be the barycentric subdivision of a regular CW-sphere $\Delta$. Plugging in $c=1$ in the $cd$-index of the face poset of $\Delta$, we get that \[\phi_\Delta(1,d)=\delta_0+\delta_1\emph{d}+\dots+\delta_{\left\lfloor \frac{n}{2}\right\rfloor} \emph{d}^{\left\lfloor \frac{n}{2}\right\rfloor}, \]
where $\delta_i$ is the sum of coefficients of monomials in $\phi_\Delta(c, d)$ for which $d$ appears $i$ times. It follows from Theorem \ref{thm: Karu's result} that $\gamma_i(\sd(\Delta))=2^i\delta_i\geq 0$. In other words, Karu's result settles Conjecture \ref{conj: Gal's conjecture} in the class of order complexes of Gorenstein$^*$ posets.

Is there a combinatorial description of the $cd$-index? The following conjecture was first proposed by Murai and Nevo \cite{Murai-Nevo} and then rephrased by Karu \cite{Karu-M vector analog}. Define the \emph{multidegree} of a degree $n$ $cd$-monomial $M$ as a zero-one vector in $\mathbb{Z}^n$ obtained by replacing each $c$ with 0 and each $d$ with 10.
\begin{conjecture}
	Let $\field$ be a field. Let $P$ be a Gorenstein$^*$ poset of rank $n+1$. Then there exists a standard $\mathbb{Z}^n$-graded $\field$-algebra $A = \oplus_v A_v$, such that for any $v\in\mathbb{Z}^n$, the coefficient of the $cd$-monomial $M$ in $\phi_P(c,d)$
with multidegree $v$ is $\dim A_v$.
\end{conjecture}	
The conjecture is known to hold for posets of Gorenstein$^*$ simplicial complexes \cite[Theorem 1.3]{Karu-M vector analog} but is open in general. 
\section{The flag upper bound conjecture}
Denote by $C_d(n)$ the cyclic $d$-polytope with $n$ vertices. The UBC for polytopes was proposed by Motzkin \cite{Motzkin-UBC} and proved by McMullen \cite{McMullen-UBT}, and it says that if $P$ is a $d$-polytope with $n$ vertices, then $f_i(P)\leq f_i(C_d(n))$ for all $1\leq i\leq d-1$. The UBC has several important generalizations: Klee \cite{Klee-Eulerian UBT} proved the UBC for all Eulerian complexes with sufficiently many vertices, Stanley \cite{Stanley-UBT} verified it for simplicial spheres, and Novik \cite{Novik-UBT} further extended it to all odd-dimensional homology manifolds and Eulerian even-dimensional homology manifolds. We remark that while the upper bounds on the face numbers for all classes of complexes mentioned above are the same, there are many polytopes or spheres that simultaneously maximize all face numbers; other classes of neighborly polytopes and neighborly spheres exist apart from the cyclic polytope. In fact, the current best lower bound for the number of combinatorial types of polytopes is attained by counting distinct neighborly polytopes \cite{Padrol-neighborly polytopes}.

The flag UBC turns out to be very different from the general UBT. Denote by $J_m(n)$ the simplicial $(2m-1)$-sphere on $n$ vertices obtained as the join of $m$ copies of the circle, each one a cycle with either $\lfloor \frac{n}{m} \rfloor$ or $\lceil \frac{n}{m}\rceil$ vertices. The following conjecture is proposed in \cite[Conjecture 6.3]{Nevo-Petersen} and \cite[Conjecture 6.3]{Nevo-Lutz}.
\begin{conjecture}
	Let $\Delta$ be a flag $(2m-1)$-sphere with $n$ vertices. Then $f_i(\Delta)\leq f_i(J_m(n))$ for all $1\leq i\leq 2m-1$. Furthermore, if $f_i(\Delta)=f_i(J_m(n))$ for some $i$, then $\Delta=J_m(n)$. 
\end{conjecture}
The above conjecture has been verified in several important cases. Adamaszek and Hladk{\'y} \cite{Adamaszek-Hladky} proved an asymptotic version of the conjecture in the class of flag triangulations of $(2m-1)$-manifolds. Using the extremal graph theory, they showed that not only the $f$-numbers, but also $h$-numbers, $g$-numbers and $\gamma$-numbers are maximized by $J_m(n)$ as long as the number of vertices is large enough. We sketch the proof below. Denote by $T_m(n)$ the clique complex generated by the complete $m$-partite Tur\'an graph with $n$ vertices and let $\Delta$ be a flag triangulated $(2m-1)$-manifold with $n$ vertices. Fix a small number $a>0$ and an integer $1\leq i<m$. If $f_i(\Delta)\leq (1-a)f_i(T_m(n))$, then since $f_i(J_m(n))=f_i(T_m(n))+O(n^i)=O(n^{i+1})$ for any $1\leq i< m$, it follows that the $n^{i+1}$ term is also the dominant term in $f_i(\Delta)$. So with $n$ large enough one can show that $f_j(\Delta)\leq f_j(J_m(n))$ for all $j$. Otherwise, since $f_{m}(\Delta)=o(n^{m+1})$, the Removal Lemma from extremal graph theory implies that a $K_{m+1}$-free subgraph $G'$ can be obtained from $G(\Delta)$ by removing only $o(n^2)$ many edges. Thus as long as there are sufficiently many vertices in $\Delta$, the number of $i$-faces in the clique complex of $G'$ is at least $(1-2a)f_i(T_m(n))$. Then by applying other extremal graph theory results, one shows that $G(\Delta)$ can be obtained from the $m$-partite Tur\'an graph by adding or deleting at most $O(n^2)$ edges. Finally since $\Delta$ triangulates a manifold, it follows that $f_i(\Delta)$ must be maximized by $f_i(J_m(n))$.

Zheng (\cite{Zheng-flag 3-dimensional}, \cite{Zheng-flag edge number}) gave a proof of the conjecture in the class of all flag 3- and 5-manifolds. Given an edge $\{u,v\}$ in a flag pseudomanifold $\Delta$, the inclusion-exclusion principle together with Lemma \ref{lm: face link prop} yields that \[f_0(\lk(u))+f_0(\lk(v))=f_0(\lk(u)\cup\lk(v))-f_0(\lk(\{u,v\}))\leq n-f_0(\lk(\{u,v\})).\]
More generally, using the inclusion-exclusion principle, one can show that for any facet $\sigma$ of $\Delta$, 
\[\sum_{v\in\sigma} f_0(\lk(v))\leq 2(m-1)n+4m.\] The upper bound then follows from the double counting. In \cite{Adamaszek-Hladky} Adamaszek and Hladk{\'y} conjectured that the same upper bounds on the face numbers continue to hold for flag odd dimensional pseudomanifolds. However, the conjectured maximizer of $f_i$ is not unique in this class: for example, both $J_2(16)$ and the join of $C_8$ and $C_4\cup C_4$ have the same $f$-vector. 

What about the flag even-dimensional upper bound conjecture? In \cite{Gal-real root conjecture} Gal showed that the roots of the $h$-polynomial of a flag homology 4-sphere $\Delta$ with $n$ vertices are all real for $d\leq 5$. As a corollary, the face numbers of $\Delta$ are bounded by those of the suspension of $J_2(n-2)$. This motivates the following conjecture. Let $J_m^*(n):=\mathcal{S}^0*C_1*\cdots*C_m$, where each $C_i$ is a circle of length either $\left\lceil{\frac{n-2}{m}}\right \rceil$ or $\left\lfloor{\frac{n-2}{m}}\right \rfloor$, and $f_0(J_m^*(n))=n\geq 4m+2$. Denote by $\mathcal{S}_{n}$ the set of flag 2-spheres on $n$ vertices, and define \[\mathcal{J}_m^*(n):=\{S*C_2*\cdots*C_m\,|\,S \in \mathcal{S}_{|V(C_1)|+2}\}.\]
\begin{conjecture}
 Let $\Delta$ be a flag homology $2m$-sphere on $n$ vertices. Then $f_i(\Delta)\leq f_i(J_m^*(n))$ for all $i=1,\dots, 2m$. If equality holds for some $1\leq i\leq 2m$, then $\Delta \in \mathcal{J}_m^*(n)$.
\end{conjecture}
The conjecture is open except for the inequality part in the $m=2$ case \cite{Gal-real root conjecture}.

\section*{Acknowledgements}
The author would like to thank Isabella Novik for helpful comments and discussions.

\end{document}